\documentclass{article}

\usepackage{amsmath}
\usepackage{amssymb}
\usepackage{amsthm}
\usepackage{nicefrac}
\usepackage{xfrac}
\usepackage{color}
\usepackage{multirow}
\usepackage[margin=4.0cm]{geometry} 
\usepackage{adjustbox}
\usepackage{mathtools}
\usepackage{tikz}
\usetikzlibrary{knots}
\usetikzlibrary{decorations.markings}
\usetikzlibrary{backgrounds}
\usepgfmodule{decorations}
\usepackage{multicol}
\usepackage{lscape}
\usepackage{tikz-cd}
\usepackage{ dsfont }
\usepackage{fancyhdr}
\theoremstyle{remark} 

\theoremstyle{plain} 
\newtheorem{theorem}{Theorem}[section]
\newtheorem{proposition}[theorem]{Proposition}

\newtheorem{question}{Question}

\newtheorem*{theorem*}{Theorem}
\newtheorem{corollary}[theorem]{Corollary}

\newtheorem{lemma}[theorem]{Lemma}
\newtheorem*{theorem1}{Theorem 1}
\newtheorem*{theorem2}{Theorem 2}
\newtheorem*{theorem3}{Theorem 3}
\newtheorem*{theorem4}{Corrolary 3.1}
\newtheorem*{theorem1b}{Theorem 1.0.1}
\usepackage{indentfirst}
\usepackage{xcolor}
\usepackage{xcolor,colortbl}

\newcommand{\Z}{\mathbb{Z}} 
\newcommand{\C}{\mathbb{C}}
\newcommand{\Q}{\mathbb{Q}} 
\newcommand{\N}{\mathbb{N}} 

\newcommand{\Ss}{\mathbf{S}}
\newcommand{\Ap}{\Delta}
\newcommand{\Apk}{\Delta_{\kappa}}
\newcommand{\Apf}{\Delta_p}
\newcommand{\Apfi}{\Delta_{p_i}}

\newcommand{\Fp}{\mathbb{F}_p}
\newcommand{\Fq}{\mathbb{F}_{p^d}}
\newcommand{\Ga}{\Gamma_K}
\newcommand{\Gab}{\Gamma_K^{\text{ab}}}
\newcommand{\Ral}{\rho_{\alpha}}
\newcommand{\Gal}{G_{\alpha}}
\newcommand{\mfb}{\mathfrak{b}}

\usepackage{blindtext}
\title{Some non-abelian covers of knots with non-trivial Alexander polynomial}
\author{Timothy Michael Morris\\ Mathematics Department, Temple University}

\begin{document}

\maketitle

\begin{abstract}
Let $K$ be a tame knot embedded in $\Ss^3$. We address the problem of finding the minimal degree non-cyclic cover $p:X \rightarrow \Ss^3 \smallsetminus K$. When $K$ has non-trivial Alexander polynomial we construct finite non-abelian representations $\rho:\pi_1\left(\Ss^3 \smallsetminus K\right) \rightarrow G$, and provide bounds for the order of $G$ in terms of the crossing number of $K$ which is an improvement on a result of Broaddus in this case. Using classical covering space theory along with the theory of Alexander stratifications we establish an upper and lower bound for the first betti number of the cover $X_\rho$ associated to the $\text{ker}(\rho)$ of $\Ss^3 \smallsetminus K$, consequently showing that it can be arbitrarily large. We also demonstrate that $X_\rho$ contains non-peripheral homology for certain computable examples, which mirrors a famous result of Cooper, Long, and Reid when $K$ is a knot with non-trivial Alexander polynomial. 
\end{abstract}

\section{Introduction}
\label{sec:Intro}

	In $1987$ Hempel \cite{Hem1} showed that the fundamental groups of Haken 3-manifolds are residually finite, i.e., $\bigcap H = \{1\}$ where $H$ ranges over the finite index normal subgroups of the fundamental group of the $3$-manifold. It follows that all topological $3$-manifolds with single a torus boundary component are residually finite. A consequence of residual finiteness is that the fundamental group admits a rich family of finite quotients, and therefore the knot manifold $M$ has an abundance of finite sheeted covers with varying topological properties.
	
	For the remainder of this paper $M_K$ always denotes the manifold $\Ss^3 \smallsetminus K$, and $\Ga= \pi_1\left(M_K \right)$. There is a very well known construction which describes an infinite family of finite covers of a knot complement, namely those which arise from the kernels of finite cyclic quotients, known as cyclic covers. Such quotients come from the following construction. Denote $\Gab = \sfrac{\Ga}{\left[\Ga,\Ga\right]}$, since $\Gab \cong \Z$,  there exists a homomorphism $\Ga \rightarrow \Z/n\Z$. The kernel of this homomorphism corresponds to a regular cover, $X_n$, typically called the $n$-fold cyclic cover of $M$.\footnote{A similar, but different notion, is the cyclic covers of $\Ss^3$ branched over the knot $K$. We do not discuss these covers.} 

when $K$ is a non-trivial knot, residual finiteness ensures the existence of non-abelian quotients of $\Ga$. Thus, there exists covers of $M_K$ which do not arise from the cyclic quotients of $H_1\left(M_K\right)$ described above. In this paper we address the following question.

\begin{question}\label{que:Q1}
What is the minimal degree non-cyclic cover of $M_K$?
\end{question} 

	The first systematic treatment of this problem was due to Broaddus \cite{Brd1}. In his thesis, he constructs explicit finite non-cyclic covers of the knot complements and provides an upper bound on the degree. Kuperberg \cite{Kup1} later decribed the growth rate of the degree of non-abelian covers as being $\text{NP}$ modulo the Generalized Reimann Hypothesis.  We improve on these results when $K$ has non-trivial Alexander polynomial.
	
	 Other than Broaddus's and Kuperberg's work there is little in the literature that directly addresses the problem of minimal degree non-cyclic covers of knot complements. Moreover, Broaddus and Kuperberg both relate the degree of the non-cyclic covers to combinatorial invariants of the knot. Let $D$ denote any diagram of $K$, recall that the crossing number of a knot is defined to be 
\begin{equation*}	
c_K=\text{min}|\{\text{Crossings \ of \ } D \}|,
\end{equation*}
where the minimum is taken over all diagrams, $D$, of the knot. Broaddus proved the following: 	 
\begin{theorem*}{(Broaddus \cite{Brd1})}
For all non-trivial knots $K$, there exists an explicit function $\mfb:\N_{3\geq} \rightarrow \N$, and there exists a finite non-cyclic cover $Z$ of $M_K$, with $[M_K:Z] \leq \mfb(c_K)$. 
\end{theorem*}
Similarly, Kuperberg proves the following result about the existence and order of finite non-abelian quotients of the group $\Ga$. In the following $\mathrm{pol}$ and $\mathrm{exp}$ represent the existence of a polynomial and exponential functions in the variable $c_K$.
\begin{theorem*}{(Kuperberg, \cite{Kup1})}
If $K$ is a non-trivial knot, then there exists a finite quotient $G$ of $\Ga$ with 
\begin{equation*}
|G|=\mathrm{exp}(\mathrm{exp}(\mathrm{pol}(c_K))).
\end{equation*}
Assuming the Generalized Reimann Hypothesis, one has
\begin{equation*}
|G|=\mathrm{exp}(\mathrm{pol}(c_K)).
\end{equation*}
\end{theorem*}

	We improve this result in terms of the degree of the cover in both results above, and drop the reliance on the Generalized Reimann Hypothesis, for a knot $K$ with non-trivial Alexander polynomial. Explicitly, we establish an upper bound similar to the result of Broaddus, however our construction yields a computationally simpler bound, in the sense that the lower bound established by Broaddus exceeds computational capability of current computer software on a standard desktop computer even for $c_K=3$. Furthermore the bound we establish is of the class $\mathrm{exp}(\mathrm{pol}(c_K))$, however both $\mathrm{exp}$ and $\mathrm{pol}$ are explicitly given. In section $2$ we prove the following Theorem.
	
\begin{theorem1}
If $K$ is a knot with crossing number $c_K$ and non-trivial Alexander polynomial, then there exists a regular non-abelian cover $X_{\Ral}$ of $M_K$ with 
\begin{equation*}
\left[M_K : X_{\Ral}\right] \leq 2^{4c_K^2}, 
\end{equation*}
and there exists an irregular non-cyclic cover $Y_{\Ral}$ with
\begin{equation*}
\left[M_K : Y_{\Ral}\right] \leq 2^{2c_K^2}. 
\end{equation*}
\end{theorem1}

	Notice that Theorem 1 addresses the minimality of \emph{regular} non-cyclic covers, providing explicit constructions and bounds. This has not been previously studied in the literature. We we strengthen the conclusion of Theorem 1 for certain important families of knots. 
	
\begin{theorem1b}
\
\begin{enumerate}
\item If $K$ is a twist knot with $2n$ half twists and non-trivial Alexander polynomial, then
\begin{equation*}
\left[M_K : Y_{\Ral}\right] \leq 16n^2. 
\end{equation*}
\item If $K$ is a fibered knot, with non-trivial Alexander polynomial we have
\begin{equation*}
\left[M_K : Y_{\Ral}\right] \leq 2^{c_K}. 
\end{equation*}
\item If $K$ is a pretzel knot $K_{p,q,r}$ for odd numbers $p,q,r$ assume $|p|$ is largest, it follows that
\begin{equation*}
\left[M_K : Y_{\Ral}\right] \leq 4p^2. 
\end{equation*}
\item For knots with Alexander polynomial of degree $n$ it follows that $n \leq c_K-1$, we have
\begin{equation*}
\left[M_K : Y_{\Ral}\right] \leq 2^{2n^2}. 
\end{equation*}
\end{enumerate}
\end{theorem1b}
For many examples the bound provided by Theorem $1$ is much larger than needed.

 	 As we have mentioned, the Alexander polynomial is a well known invariant of the knot group, denoted $\Ap(t)$, defined in $1923$ by J.W. Alexander \cite{Ale1}. Since then many authors have formulated equivalent definitions of the Alexander Polynomial (\cite{Ale1}, \cite{Fox1}, \cite{deR1}, \cite{Rolf1}, \cite{Neu1}). In order to prove Theorem $1$ we generalize the construction due to de Rahm \cite{deR1}, which which simultaneously defines the Alexander polynomial and constructs representations to affine-linear groups over $\C$. Using this point of view we are able to construct explicit, finite, metabelian representations of the knot group by generalizing de Rham's construction to an arbitrary finite field.
 	
\begin{theorem2}
There exists a homomorphism $\Ral:\Ga \rightarrow \mathrm{GL}_2\left(\Fp(\alpha)\right)$ for $p$ a prime if and only if $\alpha$ is a non-zero root of $\Ap(t) \pmod{p}$ in some finite extension of $\Fp$. This representation satisfies:

\begin{itemize}
\item $\Ral\left( \Ga \right)$ is metabelian, in particular non-abelian.
\item $|\Ral\left( \Ga \right)|=np^d$, where $n=\mathrm{ord}_{\Fp^*(\alpha)}(\alpha)$ and $d=\left[\Fp(\alpha):\Fp\right]$.  
\end{itemize}
\end{theorem2}

The quotients in Theorem $2$ being metabelian should come as no suprise. The group $\sfrac{\Ga}{\Ga''} \cong \Z \ltimes \sfrac{\Ga'}{\Ga''}$, is a metabelian group, and such finite metabelian quotients of the knot group have been extensively studied. Fox, Artin, Hartley, and Neuwirth are the pioneers in the study of metabelian covers of knots. Fox \cite{Fox4}, \cite{Fox5} describes the fundamental group of the branched cover corresponding to metacyclic representations for doubled knots. M. Artin \cite{Art1} computed the first homology groups for the same covers described in \cite{Fox5} in his senior thesis at Princeton. R. Hartley \cite{Hart1} provided a necessary and sufficient criterion for a knot to admit a finite quotient to a specific class of metabelian groups; this criterion is given in terms of the abelianization of the fundamental group of the finite cyclic covers. Lastly L. P. Neuwirth \cite{Neu1} provided a criterion in terms of the Alexander polynomial similar to what we will describe to ensure that a knot group surjects to a metacylic group. More recently, a general study of metabelian representations to $\mathrm{SL}(n,\C)$ has been a fruitful area; for example see \cite{Fri1}, \cite{BF1}, \cite{Jeb1}.

	We then turn our attention to the topological properties of the regular covers $X_{\Ral}$. From a computational point of view, the construction of $X_{\Ral}$ provides us with a large family new manifolds to examine and draw new intuition from. Their are many questions to address with regards to these regular covers. For this paper we focus on the groups $H_1\left(X_{\Ral}\right)$, in particular the computation of $\beta_1\left(X_{\Ral}\right)$. .
	
	In an homage to Thurston's work on the virtual properties of $3$-manifolds, Ian Agol's \cite{AVP1} 2014 ICM address highlighted the current state of the art for determining those properties of $3$-manifolds. His address was focused on establishing a connection between results of Haglund and Wise and current geometric methods to answer 4 of Thurston's list of 24 problems involving virtual properties of $3$-manifolds. One question, in particular, of Thurston's involved the virtual first betti number. The virtual first betti number is defined to be 
\begin{equation*}	
	v\beta_1(M)= \text{sup}\{\beta_1(\widehat{M}) \ | \ \widehat{M} \rightarrow M \text{ is a finite cover}\},
\end{equation*}	
	 Thurston asks the question: Can a closed aspherical $M$ have $v\beta_1(M) =\infty$? Agol goes on to answer this question in the positive, a consequence of the Virtual Haken, and Virtual Fibering theorems for closed manifolds. However for manifolds, $M$ with non-empty incompressible boundary it is a consequence of the The Siefert Fiber Theorem, The Torus Theorem, and the facts about peripheral subgroup separation that $v\beta_1(M)=\infty$. Furthermore the seminal paper of D. Cooper, D. Long, and A. Reid \cite{CLR1} from 1997 showed that for bounded $3$ manifolds ``non-peripheral" homology becomes unbounded in finite covers, we hope to provide an explicit construction of this statement.
	 
	 There is an extensive understanding of the topological and algebraic properties of finite cyclic covers of knot complements. In particular, complete information of the first homology groups of the cyclic covers $X_n$ can be determined directly from the Alexander polynomial of the knot $K$. Ralph Fox \cite{Fox3} using his free differential calculus showed that both the free rank (the first betti number) and the order of the torsion subgroup of $H_1(X_n)$ can be directly computed from the Alexander polynomial. In particular, the first betti number of the $X_n$ is $1$ except when $\Ap(t)$ an $n$th root of unity as a root. An immediate consequence of this is that $\beta_1(X_n) \leq \text{deg}(\Ap(t))+1$, for any $n$. The results of \cite{CLR1} provide the existence of covers with arbitrarily large betti number, and by such covers cannot be the cyclic covers of a knot complement.
	 
	We then turn our attention to the computation of $\beta_1(X_\alpha)$, as a first step in understanding such covers. Also by doing this we provide an alternate proof of corollary $1.4$ in \cite{CLR1} in the case of a knot complement with non-trivial Alexander polynomial.
	 	  
\begin{theorem3}\label{thm:thm3}
Let $p$ be the minimal prime such that $\Ap(t) \pmod{p}$ is non-trivial, $\alpha$ a root of $\Ap(t) \pmod{p}$ with $d=[\Fp(\alpha):\Fp]$ and $\mathrm{ord}_{\Fq^*}(\alpha)=n$, then the covers $X_{\Ral}$ satisfy
\begin{equation*}
 p^d-1+\beta_1(X_{n}) \leq \beta_1(X_{\Ral}) \leq (n(c_K-1))(p^d-1)+\beta_1\left(X_{n}\right).
\end{equation*}   
\end{theorem3}

The lower bound is a direct computation of the number of boundary components of the cover $X_{\Ral}$, along with basic facts about finite covering spaces. The upper bound here is a consequence of E. Hironaka's theory of Alexander stratifications and jumping loci \cite{Hir1}.

An immediate corollary of Theorem $\ref{thm:thm3}$ is:
\begin{theorem4}
When $K$ is a knot with non-trivial Alexander polynomial,
\begin{equation*}
v\beta_1(M_K) = \infty.
\end{equation*}
\end{theorem4} 

	 We will show in $\S 5$ there are knots for which $X_\alpha$ has non-peripheral elements in first homology suggesting that our methods often lead to concrete constructions of covers whose non-peripheral first homology becomes arbitrarily large. Which suggests that the covers $X_{\Ral}$ might provide a concrete construction to the famous result of Cooper, Long, and Reid in this case.  
	
	Finally we also study torsion in the first homology groups of these covers. Since $H_1(X; \Z)$ is a finitely generated abelian group, for any compact manifold $X$ is isomorphic to the group $\Z^{\beta_1(X)} \oplus T(H_1(X;\Z))$, here $T(H_1(X;\Z ) )$ is the torsion supgroup. The study of $T(H_1(N_j;\Z ) )$ for finite sheeted covers $N_j$ of a 3-manifold, $N$ is recently of significant interest. Fox's results \cite{Fox3} include an explicit formula for the order of the torsion subgroup of $T(H_1(X_n,\Z))$. This has lead to many results describing the growth of torsion in finite cyclic covers. In particular  Gordon \cite{Gord1} showed linear growth in the torsion subgroup of $H_1(X_n)$ as $n \rightarrow \infty$ for infinite classes of knots. Independently Riley \cite{Ril1}, Gonzalez-Acu\~{n}a and Short \cite{GAS1} , and Weber \cite{Web1} were able to build on Gordon's work to show exponential growth of the order of torsion through the cyclic covers of a non-trivial knot complement. 
	
	The torsion subgroup of $H_1(N_j)$ is of particularly importance when covers $N_j$ arrange into a tower of covers in the following way:
\begin{equation*}
\cdots \rightarrow N_j \rightarrow \cdots \rightarrow N_1 \rightarrow N
\end{equation*}
so that $N_i \rightarrow N_{i-1}$ is finite sheeted for all $i$. Recent work of H. Baik, D. Bauer, I. Gekhtman, U. Hamenst\"{a}dt, S. Hensel, T. Kastenholz, B. Petri, and D. Valenzuela \cite{BBGHHKPV} showed that exponential torsion growth is a generic property of random 3-manifolds. Furthermore, when the 3-manifold is endowed with a hyperbolic metric and $\bigcap_{i=1}^{\infty}\pi_1(N_i)=\{1\}$ the asymptotics of torsion growth is conjecture to have close relationship with the hyperbolic volume of $M$. In their seminal paper, N. Bergeron and A. Venkatesh formulate a conjecture \cite{BV1} (Conjecture $1.3$) to describe this asymptotic growth phenomenon. The only cases for which there are complete results in this direction are in the case of the cyclic covers of hyperbolic knot complements this is due to T. L\^{e} \cite{Le1} and independently J. Raimbault \cite{Raim1}, where the towers of cyclic covers are not exhaustive, but a similar behavior is exhibited. We will conclude the paper providing tables of computations and highlighting certain relationships between the torsion subgroups of $H_1(X_\alpha; \Z)$ and $H_1(X_n; \Z)$.

\section{The Alexander Polynomial}
\label{sec:alexanderpolynomial}

\subsection{Preliminaries}
\label{sec:prelim}

Let $\Ap(t)$ denote the Alexander polynomial of $K$ see \cite{Neu1} or \cite{Rolf1} for some of the many definitions. We use the notation $\Gamma'=[\Gamma,\Gamma]$, $\Gamma''=[\Gamma',\Gamma']$, and $\Gamma^{(i)}=[\Gamma^{(i-1)},\Gamma^{(i-1)}]$ for the $i^{\text{th}}$ iterated commutator subgroup. The classical definition of the Alexander Polynomial due to Alexander considers the split short exact sequence:
\begin{equation*}
1 \rightarrow \sfrac{\Ga'}{\Ga''} \rightarrow \sfrac{\Ga}{\Ga''} \rightarrow \Z \rightarrow 1.
\end{equation*}
The group $\sfrac{\Ga}{\Ga''}$ is a semi-direct product $\Z \ltimes \sfrac{\Ga'}{\Ga''}$, hence $\sfrac{\Ga'}{\Ga''}$ is a finite $\Z[t,t^{-1}]$ module \cite{Ale1}. Alexander proved that the annihilator of the action is a principal ideal, so it is generated by a single Laurent polynomial, $\Ap(t)$. Furthermore since $\sfrac{\Ga'}{\Ga''}$ is a finitely generated $\Z[t,t^{-1}]$ module, there exists a presentation matrix $A(t)$ of rank $k$ over $\Z[t,t^{-1}]$, called the Alexander matrix. The $i$th Alexander ideal is then the principal ideal generated by the $(k-i)$ minors of $A(t)$, and therefore $\Ap(t)$ is the generator (up to multiplication by a unit of $\Z[t,t^{-1}]$) of the zeroth Alexander ideal. We denote by $\Ap_{i}(t)$ the generator of the $i$th Alexander ideal, thus $\Ap_{i}(t)$ is the $i$th invariant factor of $A(t)$.

Denote by $\iota$, the canonical ring homomorphism $\iota:\Z \rightarrow \kappa$ for any field $\kappa$, determined by
\begin{equation*}
\mathds{1}_{\Z} \mapsto \mathds{1}_{\kappa}.
\end{equation*}
We are using the convention that $\mathds{1}_R$ is the unit in the unital ring $R$. Considering the images of the coefficients of the entries of $A(t)$ under $\iota$, we denote the resulting matrix by $A_{\kappa}(t)$. Thus $A_{\kappa}(t)$ presents $\sfrac{\Ga'}{\Ga''}$ as a $\sfrac{\Z}{\text{ker}(\iota)}[t,t^{-1}]$ module. With $\kappa$ a field $\text{ker}(\iota)=p\Z$ for $p=0$ or a prime ($p$ is the characteristic of $\kappa$). Assuming that $p$ is prime is finite, and consequentially non-zero, up to field isomorphism it follows that $A_{\kappa}(t)=A_{\Fp}(t)$, and presents $\sfrac{\Ga'}{\Ga''}$ as an $\Fp[t,t^{-1}]$ module. Furthermore since $\Fp[t,t^{-1}]$ is a principal ideal domain we define $\Apfi(t)$ to be the $i$th invariant factor of $A_{\kappa}(t)$, hence $\Apf(t)=\Delta_{p_0}(t)$. We call $\Apf(t)$ the Alexander Polynomial with coefficients in $\Fq$. For $p =0$, the image is isomorphic to $\Z$ so this construction yields the classical Alexander polynomial $\Ap(t)$. In $\S 3$ we verify that $\Apf(t)$ and $\Ap(t) \pmod{p}$ are equivalent. 

\begin{proposition}[\cite{Rolf1}]\label{prop:rofp}
Let $\Ap(t)=\Sigma_{i=-k}^la_it^i$ be the Alexander polynomial of a knot $K$ and define $\mathrm{deg}(\Ap(t))=l+k$, then
\begin{itemize}
\item[1)] $\Ap(1) = \pm 1$,
\item[2)] $a_{l-j}=a_{j-k}$ for $j=0,\ldots,\frac{l+k}{2}-1$,
\item[3)] $\mathrm{deg}(\Ap(t)) \leq c_K-1$, where $c_K$ is the crossing number of $K$.
\end{itemize}
\end{proposition}
We say $\Ap(t)$ is trivial if $\Ap(t)=\pm t^n$.
\begin{lemma} \label{lem:2.2}
Suppose $\Ap(t)$ is non-trivial, then it has at least $3$ non-zero coefficients.
\end{lemma}
\begin{proof}
Suppose $\Ap(t)$ has one non-zero coefficient $a_0$. Then by~\ref{prop:rofp}$(1)$, $\Ap(1)=a_0=\pm1$, hence $\Ap(t)$ is trivial. Thus, $\Ap(t)$ has at least $2$ non-zero coefficients $a_0$ and $a_1$ and by $2)$ they must be equal however, since $\Ap(1)=\pm2a_0$ which cannot be $1$, and the lemma follows.  
\end{proof} 
      
 We, however, bring attention to a definition of $\Ap(t)$ due to de Rham \cite{deR1}. The definition of de Rham is of particular importance to us because it allows us to simultaneously define $\Apf(t)$ and construct non-abelian representations to finite groups.  

\subsection{de Rham's Construction}
\label{sec:deRham}
 
	For a field $\kappa$ de Rham's construction begins by attempting to define a homomorphism, 
\begin{equation*}	
\varphi: \Ga\rightarrow\text{Aff}(\kappa) = \{tz+x \ | \ t \in \kappa^*, \ x \in \kappa \}.
\end{equation*} 
Using the Wirtinger presentation for $\Ga$, we have that given a diagram of $K$ with $n$ crossings
\begin{equation} \label{eq:wirt}
\Ga = \langle x_1, x_2, \ldots x_n \ \mid \  x_ix_{j(i)} = x_{j(i)}x_{i+1}  \text{ \ or \ }  x_{i+1}x_{j(i)} = x_{j(i)}x_{i} \rangle. 
\end{equation}  
It is important to note that this is a balanced presentation (the number of generators is equal to number of relators). See Figure 1 for the definition of $x_{j(i)}$ and the relations in $\Ga$.

\begin{figure}[h]
\centering
\begin{tikzpicture}[scale=1.75,knot gap=10,framed]
\draw[decoration={markings, mark=at position 1 with {\arrow[blue,thick]{>}}},postaction={decorate}, thin, knot=blue](-1,.25)--(-.25,1); 
\draw (-1,.65) node {$x_{j(i)}$};
\draw (1,-.65) node {$x_{j(i)}$};
\draw (-1,-.65) node {$x_{i}$};
\draw (1,.65) node {$x_{i+1}$};
\draw (0,-1.5) node {$(+)$};
\draw[decoration={markings, mark=at position 1 with {\arrow[blue,thick]{>}}},postaction={decorate}, thin, knot=blue](.25,-1)--(1,-.25);
\draw[decoration={markings, mark=at position 1 with {\arrow[blue,thick]{>}}},postaction={decorate}, thin, knot=blue](-.25,-1)--(-1,-.25);
\draw[decoration={markings, mark=at position 1 with {\arrow[blue,thick]{>}}},postaction={decorate}, thin, knot=blue](1,.25)--(.25,1);
\draw[decoration={markings, mark=at position 1 with {\arrow[black,thick]{>}}},postaction={decorate}, thick, knot=black](-1,-1)--(1,1);
\draw[decoration={markings, mark=at position 0 with {\arrow[black,thick]{<}}},postaction={decorate}, thick, knot=black](-1,1)--(1,-1);
\begin{scope}[xshift=4cm]
\draw[decoration={markings, mark=at position 0 with {\arrow[blue,thick]{<}}},postaction={decorate}, thin, knot=blue](-1,.25)--(-.25,1); 
\draw (-1,.65) node {$x_{j(i)}$};
\draw (1,-.65) node {$x_{j(i)}$};
\draw (-1,-.65) node {$x_{i}$};
\draw (1,.65) node {$x_{i+1}$};
\draw (0,-1.5) node {$(-)$};
\draw[decoration={markings, mark=at position 0 with {\arrow[blue,thick]{<}}},postaction={decorate}, thin, knot=blue](.25,-1)--(1,-.25);
\draw[decoration={markings, mark=at position 1 with {\arrow[blue,thick]{>}}},postaction={decorate}, thin, knot=blue](-.25,-1)--(-1,-.25);
\draw[decoration={markings, mark=at position 1 with {\arrow[blue,thick]{>}}},postaction={decorate}, thin, knot=blue](1,.25)--(.25,1);
\draw[decoration={markings, mark=at position 1 with {\arrow[black,thick]{>}}},postaction={decorate}, thick, knot=black](-1,-1)--(1,1);
\draw[decoration={markings, mark=at position 0 with {\arrow[black,thick]{<}}},postaction={decorate}, thick, knot=black](-1,1)--(1,-1);
\end{scope}
\end{tikzpicture}
\caption{The two cases for relations in $\Ga$.}
\end{figure}
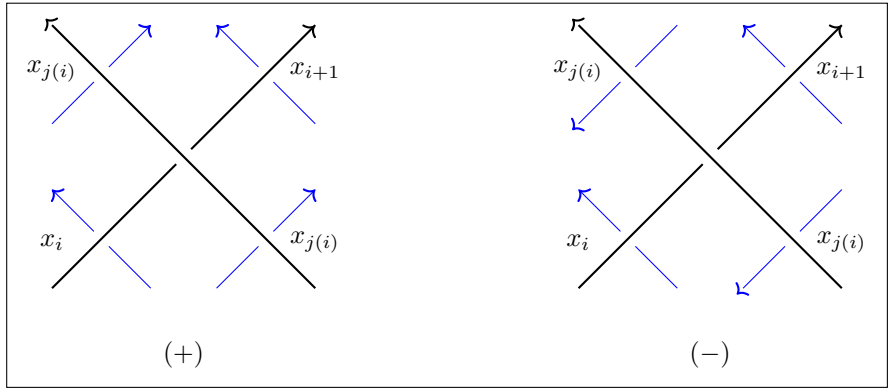

Now, $\varphi: \Ga \rightarrow \text{Aff}(\kappa)$ is a homomorphism if and only if $x_k \mapsto t_kz+y_k$ with $t_k \in \kappa^*$ and $y_k \in \kappa$ for $1\leq k \leq n$. There are two equations that could hold, one coming from each case of the relations: 
\begin{equation}\tag{$+$}
t_it_{j(i)}z+t_iy_{j(i)}+y_i = t_{j(i)}t_{i+1}z+t_{j(i)}y_{i+1}+y_{j(i)}
\end{equation}
\begin{equation}\tag{$-$}
t_{i+1}t_{j(i)}z+t_{i+1}y_{j(i)}+y_{i+1} = t_{j(i)}t_{i}z+t_{j(i)}y_{i}+y_{j(i)}
\end{equation}
Analyzing the the coefficient of $z$ we have; $t_it_{j(i)}-t_{j(i)}t_{i+1}=0$, hence $t_i=t_{i+1}$ for all $i$, renaming $t_k:=t$ for all $k$. The equations simplify to
\begin{equation}\tag{$+$}
(t-1)y_{j(i)} + y_i - ty_{i+1} = 0
\end{equation}
\begin{equation}\tag{$-$}
(t-1)y_{j(i)}-ty_{i}+y_{i+1} = 0.
\end{equation}

Let $A_{\kappa}(t) \in \mathrm{Mat}_{n\times n}\left(\kappa[t,t^{-1}]\right)$ be the presentation matrix for the equations above. The matrix $A(t)$ will denote the above presentation matrix when $\kappa$ has characteristic $0$, and in the case of de Rham $\kappa=\C$. The polynomial $\Apk(t) \in \kappa[t,t^{-1}]$ is defined to be the largest invariant factor of $A_{\kappa}(t)$. In fact $\Apk(t)$ can be taken to be a polynomial in $\kappa[t]$; we will often use this fact without explicitly stating it. Thus we are able to conclude that there exists a homomorphism $\varphi:\Ga \rightarrow \text{Aff}(\kappa(\alpha))$ if and only if $\alpha$ is a root of $\Apk(t)$ in some finite extension of $\kappa$ as described above for some non-zero $(y_1,y_2,\ldots,y_n) \in \kappa(\alpha)^n$. In other words $(y_1,y_2,\ldots,y_n)$ is a non-zero vector contained in the kernel of $A_{\kappa}(\alpha)$.

Suppose that $\alpha$ is a root of $\Apf(t)$ and $\alpha \in \Fp(\alpha)$, here $d=[\Fp(\alpha):\Fp]$. It follows from the above discussion that there exists a homomorphism $\varphi:\Ga \rightarrow \mathrm{Aff}(\kappa(\alpha))$. We will simplify notation and denote the homomorphism $\varphi$ by $\Ral$ to indicate that this homomorphism only depends on the root $\alpha$ of $\Apf(t)$.

\begin{proposition} \label{prop:2.3}
Suppose that $\alpha \in \Fq$ is a non-zero root of $\Apf(t)$, then 
\begin{equation*}
\Ral(\Ga) \cong \Gal = \langle\alpha\rangle \ltimes \Fq \leq \mathrm{Aff}(\Fq).
\end{equation*}
\end{proposition}

\begin{proof}
We first consider an alternate presentation of $\Ga$, using the presentation $(12)$, denote by $R_j$ the relations for $\Ga$. The new generating set is defined to be $\{s_i\}_{i=1}^n$, with $s_i:=x_ix_1^{-1}$ for $i\neq 1$ and $s_1=x_1$. New relations, $R'_j$, are formed from the relations $R_j$ by setting $R'_j(s_1,\ldots,s_n):=R_j(s_1,s_2x_1,\ldots,s_nx_1)$. With this presentation we have that each $s_i$ for $i \geq 2$ is an element of $\Ga'$, since the image $[s_i]$ of $s_i$ in $H_1\left( M_K \right)$ is $[x_i]-[x_1]=0$.

We have $\Ral(s_1)=\alpha z+y_1$ and we may assume that up to conjugation in $\text{Aff}(\Fq)$, that $y_1=0$. However for $i \neq 1$, since $s_i \in \Ga'$ we have 
\begin{equation*}
\Ral(s_i) \in \Ral\left(\Ga'\right) \unlhd \{ z+y_i \in \mathrm{Aff}(\Fq) \ | \  y_i \in \Fq \}
\end{equation*}
therefore if $i \geq 2$, then $\Ral(s_i)=z+y_i$ with all $y_i \in \Fq$. By construction of $\Ral$ there is a non-zero vector $(y_1,\ldots,y_k)$ contained in the kernel of $A_{\Fq}(\alpha)$, so we may assume that $y_j$ is non-zero. Furthermore by definition
\begin{equation*}
\{\Ral(s_1^{k}s_js_1^{-k})\}_{k=0}^d=\{z+\alpha^k y_j\}_{k=0}^{d}.
\end{equation*}
Now $s_j$ and $s_1^ks_js_1^{-k}$ have infinite order in $\Ga$, furthermore the image $z+y_j$ and $z+\alpha^k y_j$ have additive order $p$. We have that $\alpha^k y_j \neq 0$ for all $0 \leq k \leq d$. Also $\alpha^k y_j \neq \alpha^l y_j$ for all $k \neq l$ with $0 \leq k \leq d$ and $0 \leq l \leq d$, otherwise if $\alpha^k y_j = \alpha^l y_j$, then without loss of generality assume $k>l$, so that $\left(\alpha^{k-l}-1 \right)y_j=0$, however this cannot be the case because $\mathrm{deg}_{\Fp}(\alpha) < \mathrm{ord}_{\Fq^*}(\alpha)$, therefore $\Ral(\Ga')=\Fq$. All that is left to do is to determine the image of the powers of $s_1$, but these are precisely the affine maps $\alpha^kz$ for $0 \leq k \leq \text{ord}_{\Fq^*}(\alpha)$, we conclude that $\Ral(\Ga) \cong \langle \alpha \rangle \ltimes \Fq$.          
\end{proof} 

For the rest of this paper the image $\Ral(\Ga)$ is denoted by $\Gal$, it is clear that $\Gal \cong \langle \alpha \rangle \ltimes \left(\sfrac{\Z}{p\Z}\right)^d$, and that $|\Gal|=\text{ord}_{\Fq^*}(\alpha)p^d$.

\subsubsection*{Example: The Figure $8$ Knot}
\label{sec:F8}

We begin the fiber bundle presentation for the figure $8$ knot complement. We use this specific presentation of the figure $8$ because defining a homomorphism such as in de Rham's construction does not depend on the presentation of the fundmental group. 
\begin{equation*}
\Ga = \langle t, x, y \ | \ txt^{-1}=xyx, tyt^{-1}=yx \rangle
\end{equation*} 

Let $p$ be a prime number and $d$ some positive integer; it will become clear what $d$ should be by the end of this computation. There is a homomorphism $\Ral: \Ga \rightarrow \mathrm{Aff}(\Fq)$ if and only if 
\begin{eqnarray*}
t &\mapsto & \alpha_0 z + b_0,\\
x &\mapsto &  \alpha_1 z + b_1,\\
y &\mapsto & \alpha_2 z + b_2
\end{eqnarray*}

We may assume up to conjugation of the image of $\Ral$ in $\mathrm{Aff}(\Fq)$ that $b_0=0$. Furthermore since $\Ga'= \langle x, y \rangle$ and $\mathrm{Aff}(\Fq)'=\{ z+b \ | \ b \in \Fq \}$ it follows that $\alpha_1=\alpha_2=1 \in \Fq^*$, so we rename $\alpha_0=\alpha$. We have 
\begin{eqnarray*}
t &\mapsto & \alpha z,\\
x &\mapsto &  z + b_1,\\
y &\mapsto &  z + b_2
\end{eqnarray*}

From the first relation, $txt^{-1}=xyx$, we have 
\begin{eqnarray*}
\alpha z+\alpha b_1 &=& \alpha z+2b_1+b_2,\\
0 &=& (2-\alpha )b_1+b_2.
\end{eqnarray*}

From the second relation, $tyt^{-1}=yx$, we have 
\begin{eqnarray*}
\alpha z+\alpha b_2 &=& \alpha z+b_1+b_2,\\
0 &=& b_1+(1-\alpha)b_2.
\end{eqnarray*}
It is important to note that the coefficients of the $b_1,b_2$ are elements of $\Fq$.

Now we have the presentation matrix for the above relations, viewed as a matrix over $\Fq[\alpha,\alpha^{-1}]$
\begin{equation*}
\begin{pmatrix}
2-\alpha & 1 \\
1   & 1-\alpha
\end{pmatrix}
\begin{pmatrix}
b_1 \\
b_2
\end{pmatrix}
=
\begin{pmatrix}
0 \\
0
\end{pmatrix}
\end{equation*} 
Therefore the above equation holds for
\begin{equation*}
\begin{pmatrix} 
b_1 \\ 
b_2
\end{pmatrix} 
\neq 
\begin{pmatrix} 0 \\ 
0
\end{pmatrix}
\end{equation*}
if and only if the largest invariant factors of the matrix 
\begin{equation*}
\begin{pmatrix}
2-\alpha & 1 \\
1   & 1-\alpha
\end{pmatrix}
\end{equation*}
generate the zero ideal. The largest invariant factor is $(\Apf(\alpha))=(\alpha^2-3\alpha+1)$  this is the zero ideal if and only if $\alpha$ is a non-zero root of $\Apf(t)$. The positive integer $d$ is then seen to be the degree of the extension $\sfrac{\Fp(\alpha)}{\Fp}$.

Suppose that $p=11$, then we have that $\Apf(t)=(t-5)(t-9)$. Let $\alpha=5 \in \mathbb{F}_{11}$, then it follows that the homomorphism $\Ral$ is completely described in the following way.
\begin{eqnarray*}
t &\mapsto & 5z,\\
x &\mapsto & z+1,\\
y &\mapsto & z+3 
\end{eqnarray*}
as elements of $\mathrm{Aff}\left(\mathbb{F}_{11}\right)$.

\subsection{Fox's Free Differential Calculus and Theorem 2}
\label{ref:FoxCalc}

We investigate a defintion of $\Ap(t)$ due to Fox \cite{Fox1}. This description is given in \cite{Hir1}, and we recall it here to expand on the details and adapt the definition to allow the derivative to take coefficients in an arbitray field $\kappa$. In this section we will explicitly show that the definition of $\Ap(t)$ due to de Rham agrees with the classical definition of $\Ap(t)$. We conclude with the proof of Theorem $2$.

Let $\Lambda_r(\Z)=\Z [t_1,t_1^{-1},\ldots,t_r,t_r^{-1}]$, the ring of integral Laurent polynomials in $r$ variables. The Fox derivative can be defined in the following way. Suppose $F_r$ is the free group on $r$ generators, and $\text{ab}: \Ga \rightarrow \Ga^{\text{ab}}$ is the canonical abelianizing homomorphism. Define the mapping $D_i:F_r \rightarrow \Z[F_r]$, to be
\begin{eqnarray*}
D_i(x_j)&=&\delta_{ij}\\
D_i(uv)&=&D_i(u)+uD_i(v).
\end{eqnarray*} 

The map $\text{ab}:F_r \rightarrow F_r^{\text{ab}}$ induces a mapping $(D_1,\ldots,D_r):F_r \rightarrow \Lambda_r(\Z)^r$ which we call the \emph{Fox Derivative}, and the $D_i$ are the $i$th partials. If we define the knot group $\Ga$ in the following way,
\begin{equation*}
\Ga=\langle F_r \ | \ \mathcal{R}_j \text{ for } j=1,\ldots,s \rangle 
\end{equation*}
we obtain the following short exact sequence
\begin{equation*}
1 \rightarrow F_s \rightarrow F_r \xrightarrow[]{q} \Ga \rightarrow 1.
\end{equation*}
Let $q^*$ be the induced mapping $q^*:\Z[F_r^{\text{ab}}] \rightarrow \Z[\Ga^{\text{ab}}]$. 
We are able to form the Alexander Matrix of $r \times s$ partials also known as the Jacobian of $\Ga$ 
\begin{equation*}
M(F_r,\mathcal{R}_j)=[q^*D_i(\mathcal{R}_j)].
\end{equation*}
We may evaluate the entries of $M(F_r,\mathcal{R}_j)$ by the canonical homomorphism $\iota:\Z \rightarrow \kappa$, and the resulting matrix $M(F_r,\mathcal{R}_j)_{\Fp}$ has entries which lie in $\kappa[\Ga^{\text{ab}}]$.

\begin{theorem}[Fox \cite{Fox1}]
The $i$th Alexander ideal $(\Ap_i(t))$ is the ideal generated by the $(r-i)\times (r-i)$ minors of $M(F_r,\mathcal{R}_j)$, thus $\Ap(t)$ is the largest invariant factor of $M(F_r,\mathcal{R}_j)$.
\end{theorem}

Using the presentation $(1)$ for $\Ga$ from $\S\ref{sec:deRham}$, with $\Ga^\text{ab}\cong \langle t \rangle$ we have $q^*D_i(\mathcal{R}_k)$
\begin{eqnarray*}
q^*D_i(\mathcal{R}_i) 		&=& 1 \text{ or } -t,\\
q^*D_{i+1}(\mathcal{R}_i)	&=& -t \text{ or } 1,\\
q^*D_{j(i)}(\mathcal{R}_i)  &=& t-1, \\
q^*D_{k}(\mathcal{R}_i)		&=& 0 \text{     otherwise }
\end{eqnarray*} 

Recall the presentation matrix $A(t)$, as seen in $\S\ref{sec:deRham}$ is defined over $\Z[t,t^{-1}]$. The following corollary summarize the equivalence of Fox's Jacobian and de Rham's matrix $A(t)$.  

\begin{corollary}[de Rham \cite{deR1}] \label{cor:2.1.1}
For the knot group $\Ga$ with presentation $(\ref{eq:wirt})$, $M(F_r,\mathcal{R}_j) = A(t)$.
\end{corollary}

Furthermore the generalization of Fox's Jacobian and de Rham's presentation matrix $A_{\Fp}(t)$ to the finite field $\Fp$ yields a similar result.
\begin{corollary} \label{cor:2.1.2}
For the knot group $\Ga$ with presentation $(\ref{eq:wirt})$, $M(F_r,\mathcal{R}_j)_{\Fp}=A_{\Fp}(t)$.
\end{corollary}

\begin{theorem} \label{thm:2.7}
$\Apfi(t) \equiv \Ap_i(t) \pmod{p}$
\end{theorem}

\begin{proof}
Let the knot group $\Ga=\langle F_r | \{\mathcal{R}_j\}\rangle$. Since $\Ap_i(t)$ is a principal generator for the ideal generated by the $(r-i) \times (r-i)$ minors of $M(F_r,\mathcal{R}_j)$ let $(f_1,\ldots,f_k)=(\Ap_i(t))$ for $f_j \in \Z[t,t^{-1}]$, i.e the $f_j$ are the $(r-i) \times (r-i)$ minors of $M(F_r,\mathcal{R}_j)$. Let $(q_1,\ldots,q_k) = (\Apfi(t))$, where the $q_j \in \Fp[t,t^{-1}]$ are the $(r-i) \times (r-i)$ minors of $M(F_r,\mathcal{R}_j)_{\Fp}$, lastly we have $\iota:Z \rightarrow \Fp$ is the canonical ring homomorphism. Corollaries~\ref{cor:2.1.1} and~\ref{cor:2.1.2} give us that evaluating $\iota$ at the coefficients of the entries of $M(F_r,\mathcal{R}_j)$ is the matrix $A_{\Fp}(t)$. Furthermore denote $\iota(f)$ for $f \in \Z[t,t^{-1}]$, the image of $f$ after evaluating $\iota$ on the coefficients of $f$. Similarly if $S$ is a matrix over $\Z[t,t^{-1}]$ then $\iota(S)$ denotes the matrix with entries in $\Fp[t,t^{-1}]$, having evaluated $f$ at the coefficients of the entries of $S$. Each $f_j$ comes from a determinant of a $(r-i) \times (r-i)$ sub-matrix $S_j$, and since $\iota$ is a homomorphism we have that
\begin{equation*} 
\iota(f_j)=\iota(\text{Det}(S_j))=\text{Det}(\iota(S_j))=q_j.
\end{equation*}
 Therefore the image of $(\Ap_i(t))=  (f_1,\ldots,f_k)$ under $\iota$ is $(q_1,\ldots,q_k)=(\Apfi(t))$, the image of $\Ap_i(t)$ under $\iota$ is $\Ap_i(t) \pmod{p}$. Hence we conclude that $\Apfi(t) = \Ap_i(t) \pmod{p}$.
\end{proof}

We recall for the reader that the image of $\Ral$ the group $\Gal$ constructed in $\S\ref{sec:deRham}$, as stated $\Gal \cong \langle \alpha \rangle \ltimes \Fq$ and this semidirect product is defined via multiplication by $\alpha$. 

\begin{theorem2} \label{thm:2}
There exists a representation $\Ral:\Ga \rightarrow \Gal$ if and only if $\alpha$ is a non-zero root of $\Ap(t) \pmod{p}$ for $p$ a prime, with $d=[\Fp(\alpha):\Fp]$ and $\mathrm{ord}_{\Fq^*}(\alpha)=n$.
\end{theorem2}

\begin{proof}
It follows from Theorem~\ref{thm:2} that if $\alpha$ is a root of $\Ap(t) \pmod{p}$ in the extension $\Fp(\alpha)\cong\Fq$, for $d=[\Fp(\alpha):\Fp]$, then $\alpha$ is also a root of $\Apf(t)$. Hence by Proposition~\ref{prop:2.3} such a representation $\Ral:\Ga \rightarrow \Gal$ exists if and only if $\alpha$ is a non-zero root of $\Ap(t) \pmod{p}$.
\end{proof}

\section{Proof of Theorem 1}
\label{sec:Theorem1}
In this section we provide conditions depending only on the crossing number $c_K$ on the size of the smallest prime $p$ so that $\Apf(t)$ is a non-trivial polynomial when $\Ap(t)$ is non-trivial. In particular this allows us to provide an upper bound on the index to ensure that such a cover corresponding to $\text{ker}(\Ral)$ exists. In other words we will provide a bound on the prime $p$ so that a representation $\Ral$ of $\Ga$ onto $\Gal$ exists. 

For such a representation $\Ral$ to exist, $\Apf(t)$ must be a non-constant Alexander polynomial, so that there are non-zero roots in some extension of $\Fp$. In this section we find a bound on the smallest prime in terms of the crossing number $c_K$, for which this holds. Consider the matrix $A(t) \in \mathrm{Mat}_n\left(\Z[t,t^{-1}]\right)$ in $\S\ref{sec:deRham}$, and notice that it satisfies the following criteria;

\begin{itemize}
\item[1)] The entries are in the set $\{0, 1, t, t-1\}$.
\item[2)] In each row the entries $1, t , t-1$ occur at most once, if at all. 
\item[3)] No row is the zero vector.
\end{itemize}

\begin{lemma} \label{lem:2.3}
Suppose $A(t) \in \mathrm{Mat}_n\Z[t,t^{-1}]$ is an $n \times n$, for an $n \geq 1$, is a matrix satisfying criteria 1), 2), and 3), then for any coefficient, $a$, of the determinant of $A(t)$ we have $|a| \leq 4^{n-1}$.  
\end{lemma}

\begin{proof}
We proceed by induction on the size of the matrix $A(t)$. In the bases case $n=1$, the largest coefficient $4^0=1$. As an induction hypothesis, suppose for all $k$ with $n \geq k \geq 1$ that for any matrix $A(t)$ satisfying criteria 1), 2), and 3,) a coefficient $a$ of the determinant of $A(t)$ must satisfy $|a| \leq 4^{k-1}$. Now Consider the case $k=n+1$. Denote by $B_{k-1}(t)$, $C_{k-1}(t)$, and $D_{k-1}(t)$ the $k-1 \times k-1$ cofactor corresponding to $1,-t$ and $(t-1)$ along the first row of $A(t)$, respectively. Then we have
\begin{equation*}
\mathrm{det}(A(t))=\pm \mathrm{det}(B_{k-1}(t)) \pm t \mathrm{det}(C_{k-1}(t)) \pm (t-1)\mathrm{det}(D_{k-1}(t)).
\end{equation*}
Since $B_{k-1}(t)$, $C_{k-1}(t)$, and $D_{k-1}(t)$ all satisfy criteria 1), 2), and 3), it follows by the induction hypothesis that if $b$ is any coefficient of $\mathrm{det}(B_{k-1}(t))$, $c$ is any coefficient of $\mathrm{det}(C_{k-1}(t))$, and $d$ is any coefficient of $\mathrm{det}(D_{k-1}(t))$, that $|b| \leq 4^{n-1}$, $|c| \leq 4^{n-1}$, and $|d| \leq 4^{n-1}$. Let $a$ be any coefficient of $\mathrm{det}(A(t))$, it follows from the above equation that
\begin{equation*}
|a| \leq 4^{n-1}+4^{n-1}+4^{n-1}+4^{n} =4^{n}.
\end{equation*}
The lemma follows.
\end{proof}

\begin{lemma} \label{lem:3.2}
If $p \geq 4^{c_K-2}$ and $\Ap(t)$ is non-trivial, then $\Ap(t)$ is non-trivial in $\Fp[t]$.  
\end{lemma}

\begin{proof}

The first observation is that the presentation for $\Ga$ in $(\ref{eq:wirt})$ has $c_K$ generators since, in the Wirtinger presentation there is exactly one generator for each crossing. Furthermore, $(\ref{eq:wirt})$ can be simplified to have $c_K-1$ generators. This is due to the fact that $x_{j(c_K)}x_{c_K}x_{j(c_K)}^{-1}=x_1$ or $x_{j(c_K)}^{-1}x_{c_K}x_{j(c_K)}=x_1$ in particular $x_{j(c_K)} \neq x_{c_K}$. Thus we may drop the last generator, so that there are at most $c_K-1$ generators.
Since $\Ap(t)$ is the largest invariant factor of the matrix $A(t)$, by Lemma $\ref{lem:2.2}$ there are at least three non-zero coefficients of $\Ap(t)$. Furthermore any non-zero minor computed from $A(t)$ comes from a sub-matrix $S(t)$ that satisfies criteria 1), 2), and 3). We have
\begin{equation*}
\mathrm{det}(S(t))=\sum_{i=0}^ks_it^i, \ s_i \in \Z
\end{equation*}
and
\begin{equation*}
\Ap(t)=\sum_{i=0}^da_it^i
\end{equation*}
It follows that $s_0$ and $s_k$ are non-zero because, the leading and ending coefficients of $\Ap(t)$ are non-zero and, since $\Ap(t)$ is the largest in variant factor of $A(t)$, it must divide any maximal rank non-zero minor coming from $A(t)$. Hence the absolute values of the coefficients satisfy $a_0 \mid s_0$ and $a_d \mid s_k$. Further, by Lemma $\ref{lem:2.3}$, $|s_0| \leq 4^{c_K-2}$ and $|s_k| \leq 4^{c_K-2}$, so we have $|a_0| \leq 4^{c_K-2}$ and $|a_d| \leq 4^{c_K-2}$. If $p \geq 4^{c_K-2}$, then $a_d \neq 0 \pmod p$ and $a_0 \neq 0 \pmod p$. It follows that $\Apf(t)$ is non-constant, and by Theorem $\ref{thm:2.7}$ $\Apf(t) \equiv \Ap(t) \pmod{p}$ is non-trivial.
\end{proof}

We are now ready to prove the main result of this paper:

\begin{theorem1}
If $K$ is a knot with non-trivial Alexander polynomial, then there exists a regular non-abelian cover $X_{\Ral}$ of $M_K$ with 
\begin{equation*}
\left[M_K : X_{\Ral}\right] \leq 4^{2c_K^2-c_K}, 
\end{equation*}
and there exists an irregular non-cyclic cover $Y_{\Ral}$ with
\begin{equation*}
\left[M_K : Y_{\Ral}\right] \leq 4^{c_K^2-2c_K}. 
\end{equation*}
\end{theorem1}

\begin{proof}
Let $\Ap(t)$ be the non-trivial Alexander polynomial for a knot $K$. Let $p$ be a prime such that  $4^{c_K-1} \leq p \leq 2\cdot4^{c_K-1}-2$ which exists by Bertrand's postulate. By Theorem $\ref{lem:3.2}$, it follows that $\Apf(t)$ is non-trivial, hence there exists a non-zero root $\alpha$ in some finite extension $\Fp(\alpha) \cong \Fq$. Furthermore by $\S\ref{sec:deRham}$ there exists a representation $\Ral: \Ga \rightarrow \Gal$. Let $X_{\Ral}$ be the connected covering space of $M_K$ corresponding to $\text{ker}(\Ral)$. The index $[M:X_{\Ral}]$ is $\text{ord}(\alpha)p^d$. We have the following;
\begin{eqnarray*}
[M:X_{\Ral}]&=&\text{ord}(\alpha)p^d,\\
			&\leq &(p^d-1)(p^d),\\
			&\leq &(p^{c_K-1}-1)(p^{c_K-1}),\\
			&\leq & ((2\cdot 4^{c_K-2}-2)^{c_K-1}-1)((2\cdot 4^{c_K-2}-2)^{c_K-1}).
\end{eqnarray*}
The above bound is optimal for this argument, and the theorem follows from a simplification of the above.

If we take the subgroup of $\Gal$ generated by $\alpha$ and construct the cover corresponding to $\Ral^{-1}\left(\langle \alpha \rangle \right)$, which has index
\begin{equation*}
\left[\Ga:\Ral^{-1}\left( \langle x \rangle \right)\right]=\left[\Gal:\langle \alpha \rangle \right] = p^d.
\end{equation*}
We obtain a new non-cyclic cover of $M_K$ which we denote $Y_{\Ral}$. A similar computation follows;
\begin{eqnarray*}
[M:Y_{\Ral}]&=& p^d,\\
			&\leq &(p^d),\\
			&\leq &(p^{c_K-1}),\\
			&\leq &(2\cdot 4^{c_K-2}-2)^{c_K-1}).
\end{eqnarray*}
Again the above is optimal and the theorem follows from a simplification.
\end{proof}

\subsection{Special Cases}
\label{sec:specialcases}

The bound found in Theorem $1$ is a worst possible case for an Alexander polynomial, there are many infinite families of knots for which the Alexander polynomial takes on a specific form. Similarly there are certain properties of Alexander polynomials of knots which allow us rephrase Theorem $1$ and simplify the bounds. 

\subsubsection*{Fibered Knots}

A knot $K$ is \emph{fibered} if the complement $M_K$ is a fiber bundle over the circle. In this case the fundamental group is of the form $\Ga=\langle t \rangle \ltimes \pi_1(\Sigma_g)$, where $\Sigma_g$ is the Seifert surface of $K$ arising from Seifert's algorithm. By \cite{Rolf1}, the Alexander polynomial must be monic, and it's degree is bounded about by $2g$.

\begin{corollary}
If $K$ is a fibered knot of genus $g$, with non-trivial Alexander polynomial, then for all primes $p$ the representation $\Ral:\Ga \rightarrow \Gal$ exists and
\begin{equation*}
[M_K : Y_{\Ral}] \leq 2^{2g}.
\end{equation*}
\end{corollary}

\begin{proof}
Since $\Ap(t)$ is monic it is non-trivial modulo $2$.
\end{proof}

\subsubsection*{A family of two bridge knots $J(k,l)$}

The twists knots which we denote $J(k,l)$ are a family of two bridge knots which have exactly two half-twist regions as seen in the figure below. Each region has $k$ and $l$ half twist in their respective regions, we make the assignment that a twist is positive if it a right hand twist and negative if it is a left hand twist. The twist knots are the special case $J(-2,2n)$.
\\
\begin{figure}[h]
\centering

\begin{tikzpicture}[scale=.5]
\draw (-5,4) arc (180:90:1cm);
\draw (-4,5) -- (-2,5);
\draw (-2,5) arc (90:0:1cm);
\draw (-5,4) -- (-5,-4);
\draw (-5,-4) arc (180:270:1cm);
\draw (-4,-5) -- (-1,-5);
\draw (-1,-5.1) -- (-1,-2.9);
\draw (-1,-2.9) -- (1,-2.9);
\draw (1,-2.9) -- (1,-5.1);
\draw (1,-5.1) -- (-1,-5.1);
\draw (1,-5) -- (4,-5);
\draw (4,-5) arc (270:360:1cm);
\draw (5,-4) -- (5,4);
\draw (5,4) arc (0:90:1cm);
\draw (4,5) -- (2,5);
\draw (2,5) arc (90:180:1cm);
\draw (-1.1,4) -- (1.1,4);
\draw (-1.1,4) -- (-1.1,2);
\draw (1.1,4)--(1.1,2);
\draw (-1.1,2)--(1.1,2);
\draw (-1,2) arc (360:270:1cm);
\draw (-2,1) -- (-3,1);
\draw (-3,1) arc (90:180:1cm);
\draw (-4,0) -- (-4,-2);
\draw (-4,-2) arc (180:270:1cm);
\draw (-3,-3)--(-1,-3);
\draw (1,-3)--(3,-3);
\draw (3,-3) arc (270:360:1cm);
\draw (4,-2)--(4,0);
\draw (4,0) arc (0:90:1cm);
\draw (3,1) -- (2,1);
\draw (2,1) arc (270:180:1cm);
\draw (0,-4) node {$l$};
\draw (0,3) node {$k$}; 
\end{tikzpicture}
\caption{The Knot $J(k,l)$}
\end{figure}
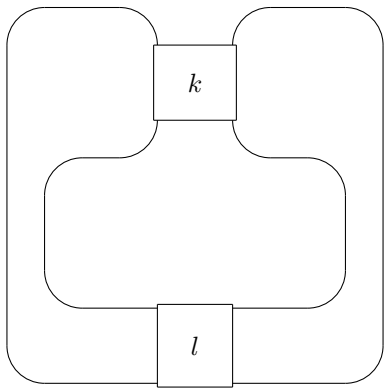

The $J(k,l)$ are knots only if $kl$ is even otherwise they are two component links. Furthermore every twist knot is isotopic to a twist knot with $l$ even.

\begin{lemma}[Lemma $7.3$ \cite{MPvL1}]\label{lem:3.4}
For all non-zero integers $k$ and $l=2n$ the knot $J(k,l)$ has Alexander polynomial 

\[
\Delta_{k,l}(t) =
\begin{cases}
nmt^2+(1-2mn)t+nm, & \text{if }k=2m \\
mt^{2n}+(1+2m)(-t^{2n-1}+\cdots-t)+m, & \text{if }k=2m+1 \text{ and } l>0\\
(m+1)t^{-2n}+(1+2m)(-t^{-2n-1}+\cdots-t)+m+1, & \text{if }k=2m+1 \text{ and } l<0\\
\end{cases}
\]

\end{lemma} 

\begin{corollary}
If $J(k,l)$ is a twist knot with $l=2n$, then for primes
\[
p \geq
\begin{cases}
mn, & \text{if }k=2m \\
m+1, & \text{if }k=2m+1 \text{ and } l=2\\
2, & \text{if }k=2m+1 \text{ and } l > 2\\
2, & \text{if }k=2m+1 \text{ and } l < 0\\
\end{cases}
\]

$\Apf(t)$ is non-trivial and the representation $\Ral:\Ga \rightarrow \Gal$ exists, and 
\[
\left[M_K : Y_{\Ral}\right] \leq
\begin{cases}
(2mn)^2-4mn+4, & \text{if }k=2m \\
(2m)^2, & \text{if }k=2m+1 \text{ and } l=2\\
2^{2n-1}, & \text{if }k=2m+1 \text{ and } l > 2\\
2^{2n}, & \text{if }k=2m+1 \text{ and } l < 0\\
\end{cases}
\]
\end{corollary}

\begin{proof}
Use the coefficients for $\Ap(t)$ from Lemma $\ref{lem:3.4}$, and compute the smallest  degree in absolute value for which the $\Apf(t)$ would have at least $3$ non-zero terms. 
\end{proof}

\subsection*{Pretzel Knots $K(p,q,r)$}

The Alexander polynomial of a pretzel knot $K(p,q,r)$, with $p,q,r$ odd numbers is known and satisfies;
\begin{equation*}
\Ap(t)=\frac{1}{4}\left((pq+qr+rp)(t^2-2t-1)+t^2+2t+1\right)
\end{equation*}

\begin{corollary}
Suppose $K(p,q,r)$ is a pretzel knot and $pq+qr+pr \neq 1$, hence $\Ap(t)$ is non trivial then there exists a non-cyclic cover $Y_{\Ral}$ and;
\begin{enumerate}
\item If $|p|$ is largest, then $\left[M_K:Y_{\Ral}\right] \leq 4p^2$.\
\item If $|q|$ is largest, then $\left[M_K:Y_{\Ral}\right] \leq 4q^2$.\
\item If $|r|$ is largest, then $\left[M_K:Y_{\Ral}\right] \leq 4r^2$.
\end{enumerate}
\end{corollary}

\subsubsection*{Knots With non-trivial Alexander polynomial}

To finish the section on special families we note that using crossing number does not provide us with a good estimate on the degree of a root of the Alexander polynomial. The degree of the Alexander polynomial is as a polynomial in $\Z[t]$ is a more accurate bound in particular if the degree of $\Ap(t)$ is $n$, then $c_K-1\geq n$ \cite{Rolf1}.

\begin{corollary}
Let $K$ be a knot with non-trivial Alexander polynomial of degree $d$, then for all primes $p > 4^{c_K-1}\geq 4^{d}$ the representation $\Ral:\Ga \rightarrow \Gal$ exists and 
\begin{equation*}
\left[M_K : Y_{\Ral}\right] \leq (2 \cdot 4^{d-1}-2)^{d}.
\end{equation*}
\end{corollary}

\section{Proof of Theorem 3}

In this section we will establish lower and upper bounds for the $\beta_1\left(X_{\Ral}\right)$, which is exactly the statement found in Theorem $3$. These bounds will be interms of the prime $p$ and properties of the root $\alpha$. 

\subsection{Lower Bound for $\beta_1(X_{\Ral})$}
\label{sec:LowB}
  
First recall that if the multiplicative order of $\alpha$ is $n$ and $[\Fp(\alpha):\Fp]=d$. Since 
\begin{equation*}
\Ga \rightarrow \Gal \rightarrow \sfrac{\Z}{n\Z} \rightarrow 1,
\end{equation*}
 we have $\text{ker}(\alpha) < \text{ker}\left(\Ga \rightarrow \sfrac{\Z}{n\Z}\right)$, so $X_{\Ral} \rightarrow X_n$ is a regular covering space with deck group $\text{ker}\left( \Gal \rightarrow \sfrac{\Z}{n\Z}\right)\cong \left(\sfrac{\Z}{p\Z}\right)^d$. We may arrange the covers in the following commutative diagram of covers. The dashed arrows denote irregular covers and the solid arrows are regular, the corresponding deck group and index denoted above the arrows.

\begin{figure}[h]
\centering
\[
\begin{tikzcd}[ampersand replacement=\&]
 \& X_{\Ral} \arrow[dl, "n"'] \arrow[ddd, "\Gal"]	 \arrow[ddr,"(\sfrac{\Z}{p\Z})^d"] 	\& 		\\
Y_{\Ral} \arrow[ddr, dashrightarrow,  "p^d"'] \&         		 		\&	 \\
	\&				\&    X_{n} \arrow[dl, "\sfrac{\Z}{n\Z}"] \\
    \& M_K		 		 			\&     	
\end{tikzcd}
\]
\caption{Commutative Diagrams of $X_{\Ral}$ and $Y_{\Ral}$.}
\end{figure}
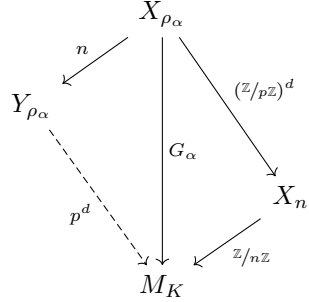

We first prove a lemma which relates classical results of cyclic covers of knot complements with the the covers $X_{\Ral}$. First we recall the famous results of Fox and Burau \cite{Fox3}.
\begin{theorem}[\cite{Fox3}]\label{thm:4.1}
If $X_n$ is the $n$-fold cyclic cover of a knot complement $M_K$, then we have the following
\begin{itemize}
\item[1)]$\beta_1(X_n)=1+|\{ \xi \in \C \ | \ \xi^n=1, \ \Ap(\xi)=0\}|$,
\item[2)]$|\mathrm{Torsion}(H_1(X_n;\Z))|=\prod_{\{ \xi \in \C \ | \ \xi^n = 1, \ \Ap(\xi) \neq 0\}}\Ap(\xi)$
\end{itemize}
\end{theorem} 

First observe that $\Ga$ may be presented as $\langle t \rangle \rtimes \Ga'$ this is a direct consequence of the split exact sequence
\begin{equation*}
1 \rightarrow \Ga' \rightarrow \Ga \xrightarrow[]{\mathrm{ab}} \Z \rightarrow 1
\end{equation*}
hence the letter $t$ can be represented by a meridian of the knot $K$. Now, since $\pi_1(X_n)\cong\mathrm{ker}(\Ga \rightarrow \sfrac{\Z}{n\Z})$ we have that $\pi_1(X_n) \cong \mathrm{ab}^{-1}(n\Z)$, and as a consequence we have
\begin{equation*}
1 \rightarrow L \rightarrow \mathrm{ab}^{-1}(n\Z) \xrightarrow[]{\mathrm{ab}|_{\mathrm{ab}^{-1}(n\Z)}} n\Z \rightarrow 1.
\end{equation*}
In the above $L=\mathrm{ker}(\mathrm{ab}^{-1}(n\Z)\rightarrow n\Z)$, thus
\begin{equation*} 
\pi_1(X_n)\cong \langle t^n \rangle \ltimes L.
\end{equation*}
When analyzing the first statement in Theorem $\ref{thm:4.1}$ the $1$ in $\beta_1(X_n)$ is exactly the contribution from the letter $t^n$, which we will call the $\emph{meridian}$ of $X_n$. Now suppose that $c \in H_1(X_n;\Z)$ is a class which generates a free factor coming from and element of $\{ \xi \in \C \ | \ \xi^n=1, \ \Ap(\xi)=0\}$. Now let $\{c_1,\ldots, c_j\}$ be the collection of all such classes, equivalently there are $j$ roots of unity which are also roots of $\Ap(t)$. It follows that
\begin{equation*}
H_1(X_n;\Q) \cong [t^n] \oplus \left(\oplus_{i=1}^jc_i\right) \cong \Z^{j+1}.
\end{equation*}
The subgroup $\left(\oplus_{i=1}^jc_i\right)$ is the subgroup of $\emph{non-peripheral}$ free homology classes of $X_n$. The classes $c_i$ will be called non-peripheral $\emph{generators}$ of $H_1(X_n;\Q)$.

In the following we have that every non-peripheral generator $c_i$ of $X_n$ will lift and generate a free summand of $H_1\left(X_{\Ral};\Q \right)$, and along with a computation of the number of boundary components we will establish a lower bound on $\beta_1\left( X_{\Ral} \right)$.  

\begin{proposition}\label{prop:1}
\begin{equation*}
\beta_1\left(X_n\right) \leq \beta_1\left(X_{\Ral}\right).
\end{equation*}
\end{proposition}

\begin{proposition}\label{prop:2}
\
The number of torus boundary components of $X_{\Ral}$ is $p^d$. 
\end{proposition}

\begin{proof}
Denote by $\Gamma_n$ by $\pi_1\left(X_n\right)$, and note that by the above there exists $p:\Gamma_n \rightarrow \left(\Z/p\Z\right)^d$. Suppose that $\langle t,\lambda \rangle$ generate the peripheral subgroup of $\Ga$, it follows that $ \langle t^n,\lambda \rangle$ generate the image of $\pi_1(\partial X_n) \hookrightarrow \Gamma_n$, where $\partial X_n$ denotes the single boundary component of $X_n$. The number of boundary components of $X_\alpha$ is equal to $\sfrac{p^d}{|\varphi(\langle t^n,\lambda \rangle)|}$. Since $\alpha(t^n)=1$ it follows that $\varphi(t^n)=1$. Since $\lambda$ bounds a Seifert surface $F$ in $M$, and hence $\pi_1(F) \hookrightarrow \Ga$ is contained in $\Ga'$ it follows that $\lambda$ also bounds the lift of $F$ to $X_n$. Therefore $\lambda \in \Gamma_n'$, $\varphi(\lambda)=1$, and we have $|\varphi(\langle t^n,\lambda \rangle)|=1$. The number of boundary components of $X_{\Ral}$ is thus $p^d$. 
\end{proof}

By the half lives-half dies Lemma \cite{Hat1} the collection of lifts of meridinal boundary curve $t^n$, denoted $\{\tilde{t^n}_1, \ldots, \tilde{t^n}_{p^d}\}$ contribute to $\beta_1\left(X_{\Ral}\right)$. Using Propositions $\ref{prop:1}$ and $\ref{prop:2}$ we obtain the following lower bound.

\begin{theorem}
If $K$ is a knot with non-trivial Alexander polynomial, and $\Ral:\Ga \rightarrow \langle \alpha \rangle \rtimes \left(\sfrac{\Z}{p\Z}\right)^d$ is the representation constructed in $\S\ref{sec:deRham}$, then
\begin{equation*}
p^d+\beta(X_n)-1 \leq \beta_1 \left( X_{\Ral} \right).
\end{equation*}
\end{theorem}

\subsection{Alexander Stratifications and The Upper Bound for $\beta_1(X_{\Ral})$}

Recall the notation for a finitely presented group $\Ga=\langle F_r \ | \ \mathcal{R}_j \text{ for } j=1,\ldots,s \rangle$, as
\begin{equation*}
1 \rightarrow F_s \rightarrow F_r \xrightarrow[]{q} \Ga \rightarrow 1.
\end{equation*} 
The character group of $\Ga$ is defined to be $\widehat{\Ga}=\text{Hom}(\Ga,\C^*)$. For any $f \in \widehat{\Ga}$, the $r \times s$ matrix $M(F_r,\mathcal{R}_j)(f)$ defined in $\S \ref{ref:FoxCalc}$ is given by evaluation by $f$. The Alexander stratification of $\widehat{\Ga}$ is 
\begin{equation*}
V_i(\Ga)=\{f \in \widehat{\Gamma}_K \ | \ \text{rank}(M(F_r,\mathcal{R}_j)(f))<r-i \}. 
\end{equation*}
The $V_i$ are the subsets of $\widehat{\Gamma}_K$ defined by the ideals of the $(r-i) \times (r-i)$ minors of $M(F_r,\mathcal{R}_j)$. The nested sequence of algebraic subset $\widehat{\Gamma}_K \supset V_1 \supset \cdots \supset V_r$ is called the \emph{Alexander Stratification of} $\widehat{\Gamma}_K$.

The reason for introducing the Alexander Stratification is to apply the following theorem.

\begin{theorem}[Hironaka \cite{Hir1}]
Suppose that $p:Y \rightarrow X$ is a covering space of connected manifolds and 
\begin{equation*}
\Large{\sfrac{\pi_1(X)}{p_*\pi_1(Y)}}=A
\end{equation*}
is a finite abelian group. Let $q:\pi_1(X) \rightarrow A$ be the quotient map and $\widehat{q}:\widehat{A}\hookrightarrow\widehat{\Ga}$ the induced inclusion map. Then
\begin{equation*}
\beta_1(Y) = \sum_{i=1}^{r-1}|V_i(\pi_1(X)) \cap \widehat{q}(\widehat{A} \smallsetminus \widehat{1})| + \beta_1(X).
\end{equation*} 
\end{theorem}

Since $A$ is a finite abelian group every element of $\widehat{A}$ is determined by a root of unity, therefore $|\widehat{A}|=|A|$. An immediate consequence of this observation and Theorem $3.2$ is the following corollary. 

\begin{corollary}
Let $X_{\Ral}$ be as above, and $X_n$ denote the $n$-fold cyclic cover of the knot complement. Then $p:X_{\Ral} \rightarrow X_n$ is a covering map with deck group $A$, an elementary abelian $p$ group, for $p$ the smallest prime so that $\Apf(t)$ is non-trivial, then
\begin{equation*}
\beta_1(X_{\Ral})=\sum_{x=1}^{r-1}|V_i(\pi_1(X_n)) \cap \widehat{q}(\widehat{A} \smallsetminus \widehat{1})| + \beta_1(X_n) \leq (r-1)(p^d-1)+\beta_1(X_n).
\end{equation*} 
\end{corollary}

This completes the proof of Theorem $3$:
\begin{theorem3}
Let $K$ be a knot with non trivial Alexander polynomial and $\Ral:\Ga \rightarrow \Gal$ is the representation constructed in $\S\ref{sec:deRham}$ for some root $\alpha\in \Fq$ with $\mathrm{ord}_{\Fq^*}(\alpha)=n$, and the positive integer $r$ is the number of generators in a presentation of $\Gamma_n$. Then
\begin{equation*}
 p^d+\beta_1\left(X_n\right)-1 \leq \beta_1(X_{\Ral}) \leq (r-1)(p^d-1)+\beta_1\left(X_{n}\right).
\end{equation*}  
\end{theorem3} 

We have established an upper bound for the betti number of $X_{\Ral}$. When $K$ is fibered we are able to improve this bound, in this case $\Ga'$ is a free group on $2g$ letters, where $g$ is the genus of $K$. Hence 
\begin{equation*}
\Ga=\langle t,x_1,\ldots, x_{2g} \ | \ tx_it^{-1}=w_i \text{ for } i=1, \ldots,2g \rangle,
\end{equation*}
the $w_i$ are words in the $x_i$. This gives us a presentation 
\begin{equation*}
\Gamma_n=\langle t^n,x_1,\ldots,x_{2g} \ | \ t^nx_it^n=u_i\rangle,
\end{equation*}
where $u_i$ is a word in the $x_j$ coming from the rule that $tx_it^{-1}=w_i$. Thus there are $2g$ relations and $2g+1$ variables. To remain consistent with our notation, we express this presentation in the following way 
\begin{equation*}
1 \rightarrow F_{2g} \rightarrow F_{2g+1} \rightarrow K_n \rightarrow 1.
\end{equation*} 
We have that have that the Fox partial $D_0(t^nx_it^nu_i^{-1})$ is $1-x_i$, where the zeroth index is regarded as the index of the generator $t^n$.

\begin{lemma}\label{lem:4.1}
Let $X_n$ be the n-fold cyclic cover of a fibered knot complement, with presentation 
\begin{equation*}
1 \rightarrow F_{2g} \rightarrow F_{2g+1} \rightarrow K_n \rightarrow 1
\end{equation*}
described above. If $p:Y \rightarrow X_n$ is a regular covering space with finite abelian deck group $A$ and $q:\pi_1(X_n) \rightarrow A$ with $t^n \in \mathrm{ker}(q)$, then 
\begin{equation*}
V_{(2g+1)-1}(\pi_1(X_n)) \cap \widehat{q}(\widehat{A} \smallsetminus \widehat{1}) = \emptyset
\end{equation*}
\end{lemma}

\begin{proof}
Assume by way of contradiction that $f \in V_{2g}(\pi_1(X_n)) \cap \widehat{q}(\widehat{A} \smallsetminus \widehat{1})$, hence $f(x)=a(q(x))$ for all $x \in K_n$ and some $a\in \widehat{A}\smallsetminus \widehat{1}$. Furthermore since $f \in V_{2g}(\pi_1(X_n))$, we have $D_0(t^nx_it^nu_i^{-1})(f)=0$ for all $i=1,\dots,2g$. Thus $f(x_i)=1$ and $x_i \in \text{ker}(q)$ for all $i=1,\dots,2g$. However $A$ is a non-trivial quotient of $K_n$, hence with $t^n \in \text{ker}(q)$ at least $1$ generator $x_j$ is not contained in the kernel of $q$. We have reached a contradiction, thus $V_{(2g+1)-1}(\pi_1(X_n)) \cap \widehat{q}(\widehat{A} \smallsetminus \widehat{1}) = \emptyset$.    
\end{proof}
 
\begin{corollary}\label{cor:4.11}
Let be $\alpha$ a root of $\Apf(t)$ of order $n$ and degree $d$ over $\Fp$, for a fibered knot of genus $g$. If $X_{\Ral} \rightarrow X_{n}$ the associated regular cover and $q:K_n \rightarrow \Fq$ then 
\begin{equation*}
\beta_1(X_{\alpha})=\sum_{x=1}^{2g-1}|V_i(\pi_1(X_n)) \cap \widehat{q}(\widehat{A} \smallsetminus \widehat{1})| + \beta_1(X_n) \leq (2g-1)(p^d-1)+\beta_1(X_n).
\end{equation*}
\end{corollary}

\begin{proof}
All that we need to show is that $t^n \in \text{ker}(q)$, however this follows directly from the fact that $t^n \in \text{ker}(\alpha)$, and $\text{ker}(q)=\text{ker}(\alpha)$.
\end{proof}

This corollary is particularly interesting when we consider the figure $8$ or trefoil knot complements, denoted $4_1$ and $3_1$ in the Rolfsen-Thistleswaithe table. These knots are fibered and have genus $1$, furthermore $\beta_1(X_n)=1$ for all $n > 1$. In this case Theorem $\ref{thm:2.7}$ and Corollary $\ref{cor:4.11}$ say that for any prime $p$ and $\alpha$ a root of $\Apf(t)$ of order $n$ and degree $d$  
\begin{equation}
p^d+\beta_1(X_n)-1 \leq \beta_1(X_{\Ral}) \leq p^d+\beta_1(X_n)-1.
\end{equation} 
Therefore, for the figure $8$, $\beta_1(X_{\Ral})=p^d$, and since $\Ap(t)=t^2-3t+1$, we have that $\beta_1(X_{\Ral})=p$ if $\Ap(t)$ factors modulo $p$ and $\beta_1(X_{\Ral})=p^2$ if $\Ap(t)$ does not factor modulo $p$. 

For the trefoil we have $\Ap(t)=t^2-t+1$, hence 
\begin{equation*}
\beta_1\left( X_{\Ral} \right) =
\begin{cases}
p+2 & \text{ if } 6 | n, \text{ and } \Apf(t) \text{ factors }\\
p^2+2 & \text{ if } 6 | n, \text{ and } \Apf(t) \text{ is irreducible }\\
p & \text{ if } 6 \nmid n, \text{ and } \Apf(t) \text{ factors }\\
p^2 & \text{ if } 6 \nmid n, \text{ and } \Apf(t) \text{ is irreducible }
\end{cases}
\end{equation*}

In particular the bound in Corollary $4.5$ is sharp, and more surprising is that this is not the only case for which it is sharp. It is sharp for $3_1$, $5_1$,  as seen below in $\S 5$, as seen in the tables below.

\section{Computations of First Homology and Questions}

The computations in this section were done using both Sagemath \cite{Sage1} and Magma \cite{Mag1}. The following table below can be interpreted in the following way, for each knot listed there will be $2$ rows, the upper row is $H_1\left(X_{\Ral}\right)$ and the lower row is $H_1\left(X_n\right)$ of the cyclic cover $X_n$ subordinate to $X_{\Ral}$. The  notation used in the table may be understood in the following way;
\begin{equation*}
\left[0^{r_0}, n_1^{r_1},\ldots,n_k^{r_k}\right]\leftrightarrow \Z^{r_0} \oplus \left(\sfrac{\Z}{n_1\Z}\right)^{r_1} \oplus \cdots \oplus \left(\sfrac{\Z}{n_k\Z}\right)^{r_k}.
\end{equation*}
Blank spaces in the table below indicate that magma timed out in the computation of the abelianization of the kernels of $\Ral$, and $\emptyset$ indicates that $\Ap(t)$ is trivial modulo $p$. 
\begin{question}\label{que:q2}
What is $|\mathrm{Torsion}\left(H_1\left(X_{\Ral};\Z\right)\right)|$?
\end{question}

It would be nice to compute the order of the torsion in terms of an invariant of the knot or the covers $X_n$ or $X_{\Ral}$, to mirror the computation of the torsion for the $X_n$. There are many examples below for which
\begin{equation}
|\mathrm{Torsion}\left(H_1\left(X_{\Ral};\Z\right)\right)|=\frac{|\mathrm{Torsion}\left(H_1\left(X_{n};\Z\right)\right)|}{p^d}.
\end{equation}
Recall that $p^d$ is the degree of the cover $X_{\Ral} \rightarrow X_n$.

For example $(3)$ holds for $3_1$ and the primes $2$ and $3$, for the other primes in the table first homology of the cyclic covers of $3_1$ is torsion free. The same behavior is seen in the knot $5_1$, for the prime $5$ this is the case and for the other primes first homology of the cyclic covers is torsion free. For the knot $4_1$ equation $(2)$ holds for each prime presented in the table below. The knot $5_2$ exhibits the same behavior, except for the prime $3$ here the order of the torsion is larger than that of the cyclic cover. In the knot $6_1$ equation $(3)$ holds for all primes in the table, much like $4_1$. The knot $6_2$ sees equation $(3)$ hold for one root of $\Ap(t) \pmod{11}$ but not the other root of the same polynomial, $6_3$ also exhibits this behavior for the primes $13$ and $7$.  

\begin{question}
What feature of a knot makes $(3)$ hold? For which primes does it hold?
\end{question}

The for the knots $3_1$, the $(3,2)$ torus knot, and $5_1$, the $(5,2)$ torus knot, the upper bound for $\beta_1\left(X_{\Ral} \right)$ described in equation $(2)$ is realized for certain values of $p$. Specifically the primes $3,7,$ and $11$ for $5_1$ and $5,7,11$, and $13$ for $3_1$. For $3_1$ there is no torsion in $H_1\left(X_{\Ral};\Z\right)$ for primes such that $\mathrm{ord}_{\Fq^*}(\alpha)$ is $6=2\cdot 3$. Similarly for $5_1$ and primes such that $\mathrm{ord}_{\Fq^*}(\alpha)$ is $10=5\cdot 2$. This phenomomenon also holds for one computable case of the $(7,2)$ torus knot, not appearing in the table below. What is even more interesting is that all the roots of the Alexander polynomial of the torus knot $T(p,q)$ are $pq$ roots of unity.
\begin{question}
If $T(p,q)$ is a torus knot, is the upper bound in $(2)$ for $\beta_1\left(X_{\Ral}\right)$ realized for certain primes $p$, and is $H_1\left(X_{\Ral};\Z\right)$ torsion free for these primes? 
\end{question}

In many cases below the lower bound for $\beta_1\left(X_{\Ral}\right)$ is realized however for the knot $6_3$ and the prime $2$ it is not, and neither is the upper bound. This phenomenon is also demonstrated in the the knot $6_2$ and the primes $2$ and $3$.  

\begin{question}
For what knots is the upper bound in Theorem $3$ or equation $(2)$ realized? similarly what knots is the lower bound of Theorem $3$ realized?
\end{question}

Another important feature of the table below is the light blue colored cells are indicating which metabelian covers $X_{\Ral}$ produce the minimal degree non-cylic cover as the quotient $Y_{\Ral}$. It can be a seen that a for a few examples $Y_{\Ral}$ is the minimal degree non-cyclic cover for the knot in Question.

\begin{question}
For the knots $5_2$, $6_2$ and $6_3$ what is the minimal degree non-cyclic cover? Is the cover somehow related to $Y_{\Ral}$.
\end{question}

\begin{question}
What feature must a knot have to make $Y_{\Ral}$ the minimal degree non-cyclic cover?
\end{question}

\newgeometry{left=.8cm}
\begin{table}[]

\footnotesize
\begin{tabular}{|l|l|l|l|l|l|l|}
\hline
                  & $2$ & $3$ & $5$ & $7$ & $11$ & $13$  \\ \hline
\multirow{2}{*}{$3_1$} & $\left[ 0^4 \right]$ & \cellcolor{blue!25}$\left[0^3 \right]$ & $\left[ 0^{27} \right]$ & $\left[ 0^{9} \right]$ & $\left[ 0^{123} \right]$ & $\left[ 0^{15} \right]$ \\ \cline{2-7} 
                  & $\left[0,2^2 \right]$ & $\left[0, 3\right]$ & $\left[0^3 \right]$ & $\left[ 0^{3} \right]$ & $\left[ 0^{3} \right]$ & $\left[ 0^{3} \right]$  \\ \hline
\multirow{2}{*}{$4_1$} & \cellcolor{blue!25}$\left[ 0^{4},2^2 \right]$ & $\left[ 0^{9}, 5 \right]$ & $\left[ 0^{5} \right]$ & $\left[ 0^{49}, 3^2, 5 \right]$ & $\left[ 0^{11}, 11 \right]$ & $\left[ 0^{169}, 5, 29^2 \right]$  \\ \cline{2-7} 
                  & $\left[ 0,4^2 \right]$ & $\left[ 0,3^2,5 \right]$ & $\left[ 0, 5 \right]$ & $\left[ 0, 3^2,5,7^2 \right]$ & $\left[ 0,11^2 \right]$ & $\left[ 0,5,13^2,29^2 \right]$  \\ \hline
\multirow{2}{*}{$5_1$} & $\left[ 0^{26} \right]$ & $\left[ 0^{245} \right]$ & \cellcolor{blue!25} $\left[ 0^5 \right]$ &  $\left[ 0^{7205} \right]$& $\left[ 0^{35} \right]$ &   \\ \cline{2-7} 
                  & $\left[ 0, 2^4 \right]$ & $\left[ 0^5 \right]$ &  $\left[ 0,5 \right]$ &  $\left[ 0^5 \right]$ & $\left[ 0^5 \right]$ &   \\ \hline
\multirow{2}{*}{$5_2$} & $\emptyset$ &$\left[ 0^9,2^4,7 \right]$  & $\left[ 0^{25} \right]$ & $\left[ 0^7 \right]$ & $\left[ 0^{11},11 \right]$ & $\left[ 0^{169} \right]$  \\ \cline{2-7} 
                  & $\emptyset$ & $\left[ 0,3^3 \right]$ & $\left[ 0, 5^2 \right]$ & $\left[ 0, 7 \right]$ & $\left[ 0,11,11 \right]$ & $\left[ 0,13^2 \right]$   \\ \hline
\multirow{2}{*}{$6_1$} & $\emptyset$ &\cellcolor{blue!25}$[0^3,3]$  & $[0^5,9,5]$ & $[0^7,7]$ & $[0^{11},9,11,31^2]$  & $[0^{13},3,27,5^2,7^2,13]$  \\ \cline{2-7} 
                  & $\emptyset$  & $[0,9]$ & $[0,9,5^2]$ & $[0,7^2]$ & $[0,9,11^2,31^2]$ & $[0,3,27,5^2,7^2,13^2]$   \\ \hline
\multirow{2}{*}{$6_2$} & $[0^{21},4]$ & $[0^{17}, 2^{8},3^2,11]$ & $[0^{25},5^3]$ &  & $[0^{11}]$; $[0^{121},5^2,11^{18},121^2,23^{24},43^{12}]$ &  \\ \cline{2-7} 
                  & $[0,2^4]$ & $[0,3^4,11]$ & $[0,5^2]$ &  & $[0,11]$; $[0,5^2,11^3]$ &   \\ \hline
\multirow{2}{*}{$6_3$} & $[0^{31},2^5,4]$ & $[0^9,3^4,13]$  &  & $[0^7,7]$; $[0^{49},3^2,13^{17}]$ &  & $[0^{13}]$; $[0^{169},5^{28},13^{38},169^4,43^2,181^{28}]$   \\ \cline{2-7} 
                  & $[0,4^4]$ & $[0,3^2,13]$  &  & $[0,7^2]$; $[0,3^2,7^2,13]$ &  &  $\left[0,13 \right]$; $\left[0,13^3,43^2\right]$  \\ \hline

\end{tabular}
\normalsize

\end{table}
\restoregeometry

\section*{Minimal Degree non-cyclic covers vs. $Y_{\Ral}$}

In what follows is a list of knots up to $7$ crossings, the number next to each knot is the minimal degree of a non-cyclic cover. The yes or no, indicates if $Y_{\Ral}$ is the minimal degree non-cyclic cover, and $\Ap(t) \pmod{p}$ is written next to this which is used to determine whether $Y_{\Ral}$ is minimal. 

\begin{itemize}
\item $3_1, 3$: Yes, $((t + 1)^2, 3)$
\item $4_1, 4$: Yes, $(t^2 + t + 1, 2)$
\item $5_1, 5$: Yes, $((t + 1)^4, 5)$
\item $5_2, 5$: No, $((2) * (t^2 + t + 1), 5)$
\item $6_1, 3$: Yes, $((-1) * (t + 1)^2, 3)$
\item $6_2, 5$: No, $((t^2 + t + 1)^2, 5)$
\item $6_3, 5$: No, $(t^4 + 2*t^3 + 2*t + 1, 5)$
\item $7_1, 7$: Yes, $((t + 1)^6, 7)$
\item $7_2, 4$: Yes, $(t^2 + t + 1, 2)$
\item $7_3, 4$: Yes, $(t * (t^2 + t + 1), 2)$
\item $7_4, 3$: Yes, $((t + 1)^2, 3)$
\item $7_6, 6$: No, $(t^2, 2), ((-1) * (t^4 + t^3 + t^2 + t + 1), 3)$
\item $7_7, 3$: Yes, $((t + 1)^4, 3)$
\end{itemize}
\bibliography{Research1}
\bibliographystyle{plain}

\end{document}